\documentclass[10pt]{article}

\usepackage{ucs}
\usepackage{amssymb}
\usepackage{amsthm}
\usepackage{amsmath}
\usepackage{latexsym}
\usepackage[cp1251]{inputenc}
\usepackage[english]{babel}
\usepackage{graphicx}
\usepackage{wrapfig}
\usepackage{caption}
\usepackage{subcaption}
\usepackage{txfonts}
\usepackage{lmodern}
\usepackage{mathrsfs}
\usepackage{algorithm}
\usepackage{multicol}
\usepackage[hidelinks]{hyperref}
\usepackage{enumerate}
\usepackage{titlesec}

\newcommand*{\myproofname}{Proof}
\newenvironment{claimproof}[1][\myproofname]{\begin{proof}[#1]}{\end{proof}}

\usepackage[left=2.7cm,right=2.7cm,top=2.7cm,bottom=2.7cm,bindingoffset=0cm]{geometry}

\title{The asymptotic behavior of the correspondence chromatic number}
\date{}
\author{Anton Bernshteyn\thanks{Supported by the Illinois Distinguished Fellowship.}\\ University of Illinois at Urbana-Champaign}

\makeatletter
\newcommand{\neutralize}[1]{\expandafter\let\csname c@#1\endcsname\count@}
\makeatother

\newtheorem{theo}{Theorem}[section]

\newtheorem{lemma}[theo]{Lemma}

\newtheorem{corl}[theo]{Corollary}

\newtheorem{claim}{Claim}[theo]

\theoremstyle{definition}
\newtheorem{defn}[theo]{Definition}

\newtheorem{remk}[theo]{Remark}

\renewcommand{\epsilon}{\varepsilon}

\newenvironment{lemmacopy}[1]
{\renewcommand{\thetheo}{\ref{#1}}%
	\neutralize{theo}\phantomsection
	\begin{lemma}}
	{\end{lemma}}

\newenvironment{theobis}[1]
{\renewcommand{\thetheo}{\ref{#1}$'$}
	\neutralize{theo}\phantomsection
	\begin{theo}}
	{\end{theo}}

\begin{document}
	
	\maketitle
	
	\begin{abstract}
		Alon \cite{Alon} proved that for any graph $G$, $\chi_\ell(G) = \Omega(\ln d)$, where $\chi_\ell(G)$ is the list chromatic number of $G$ and $d$ is the average degree of $G$. Dvo\v{r}\'{a}k and Postle~\cite{DP} recently introduced a generalization of list coloring, which they called \emph{correspondence coloring}. We establish an analog of Alon's result for correspondence coloring; namely, we show that $\chi_c(G) = \Omega(d/\ln d)$, where $\chi_c(G)$ denotes the correspondence chromatic number of $G$. We also prove that for triangle-free $G$, $\chi_c(G) = O(\Delta/\ln \Delta)$, where $\Delta$ is the maximum degree of $G$ (this is a generalization of Johansson's result about list colorings~\cite{Johansson}). This implies that the correspondence chromatic number of a regular triangle-free graph is, up to a constant factor, determined by its degree.
	\end{abstract}
	
	\section{Introduction}
	
	An important generalization of graph coloring, so-called \emph{list coloring}, was introduced independently by Vizing~\cite{Vizing} and Erd\H{o}s, Rubin, and Taylor~\cite{ERT}. It is defined as follows. Let $G$ be a graph\footnote{All graphs considered here are finite, undirected, and simple.} and suppose that for each vertex $v \in V(G)$, a set of available colors $L(v)$, called the \emph{list} of $v$, is specified. A proper coloring $c$ of $G$ is an \emph{$L$-coloring} if $c(v) \in L(v)$ for all $v \in V(G)$. $G$ is said to be \emph{$L$-colorable} if it admits an $L$-coloring; $G$ is \emph{$k$-list-colorable} (or \emph{$k$-choosable}) if it is $L$-colorable whenever $|L(v)| \geq k$ for all $v \in V(G)$. The least number $k$ such that $G$ is $k$-choosable is called the \emph{list chromatic number} (or the \emph{choosability}) of $G$ and is denoted by $\chi_\ell(G)$ (or $\operatorname{ch}(G)$).
	
	For all graphs $G$, $\chi(G) \leq \chi_\ell(G)$, where $\chi(G)$ denotes the ordinary chromatic number of $G$. Indeed, $G$ is $k$-colorable if and only if it is $L$-colorable with the list assignment such that $L(v) = \{1, \ldots, k\}$ for all $v \in V(G)$. This inequality can be strict; in fact, $\chi_\ell(G)$ cannot be bounded above by any function of $\chi(G)$ since there exist bipartite graphs with arbitrarily high list chromatic numbers.
	
	A striking difference between list coloring and ordinary coloring was observed by Alon in \cite{Alon}: It turns out that the list chromatic number of a graph can be bounded below by an increasing function of its average degree. More precisely:
	
	\begin{theo}[Alon~\cite{Alon}]\label{theo:Alon}
		Let $G$ be a graph with average degree $d$. Then
		$$
			\chi_\ell(G) \geq (1/2 - o(1))\log_2 d,
		$$
		where we assume that $d \to \infty$.
	\end{theo}
	
	In this paper we study another notion, more general than list coloring, which was recently introduced by Dvo\v{r}\'{a}k and Postle~\cite{DP} in order to prove that every planar graph without cycles of lengths $4$ to $8$ is $3$-list-colorable, answering a long-standing question of Borodin~\cite{Borodin}. Before we proceed to the actual definition, let us consider an example. Suppose that $G$ is a graph and $L$ is a list assignment for~$G$. For each $v \in V(G)$, let
	$$
		\widetilde{L}(v) \coloneqq \{(v, c)\,:\, c \in L(v)\}.
	$$
	Thus, the sets $\widetilde{L}(v)$ are pairwise disjoint. Let $H$ be the graph with vertex set
	$$
		V(H) \coloneqq \bigcup_{v \in V(G)} \widetilde{L}(v)
	$$
	and edge set
	$$
		E(H)\coloneqq \{(v_1, c_1) (v_2, c_2)\,:\, v_1v_2 \in E(G), c_1 = c_2\}.
	$$
	Given an $L$-coloring $c$ of $G$, we define the set $I_c \subseteq V(H)$ as follows:
	$$
		I_c \coloneqq \{(v, c(v))\,:\, v \in V(G)\}.
	$$
	Observe that $I_c$ is an independent set in $H$ and for each vertex $v \in V(G)$, $|I_c \cap \widetilde{L}(v)| = 1$. Conversely, if $I \subseteq V(H)$ is an independent set such that $|I \cap \widetilde{L}(v)| = 1$ for all $v \in V(G)$, then, setting $c_I(v)$ to be the single color such that $(v, c_I(v)) \in I$, we obtain a proper $L$-coloring $c_I$ of $G$.
	
	This example can be generalized as follows.
	
	\begin{defn}
		Let $G$ be a graph. A \emph{cover} of $G$ is a pair $(L, H)$, where $L$ is an assignment of pairwise disjoint sets to the vertices of $G$ and $H$ is a graph with vertex set $\bigcup_{v \in V(G)} L(v)$, satisfying the following two conditions.
		\begin{enumerate}
			\item If $xy \in E(H)$, where $x \in L(u)$, $y \in L(v)$, then $uv \in E(G)$ (in particular, $u \neq v$).
			\item For each $uv \in E(G)$, the edges between $L(u)$ and $L(v)$ form a matching.
		\end{enumerate}
	\end{defn}
	
	\begin{defn}\label{defn:coloring}
		Suppose that $G$ is a graph and $(L, H)$ is a cover of $G$. An \emph{$(L,H)$-coloring} of $G$ is an independent set $I \subseteq V(H)$ such that $I \cap L(v) \neq \emptyset$ for all $v \in V(G)$. In this context, we refer to the vertices of $H$ as the \emph{colors}. $G$ is said to be \emph{$(L,H)$-colorable} if it admits an $(L,H)$-coloring.
	\end{defn}
	\begin{remk}
		Note that Definition~\ref{defn:coloring} allows for more than one color to be used at a given vertex of $G$. However, if $G$ is $(L,H)$-colorable, then we can always find an $(L,H)$-coloring that uses exactly one color for each vertex.
	\end{remk}
	
	\begin{defn}
		A graph $G$ is \emph{$k$-correspondence-colorable} (\emph{$k$-c.c.} for short) if it is $(L, H)$-colorable whenever $(L,H)$ is a cover of $G$ and $|L(v)| \geq k$ for all $v \in V(G)$. The least number $k$ such that $G$ is $k$-c.c. is called the \emph{correspondence chromatic number} of $G$ and is denoted by $\chi_c(G)$.
	\end{defn}
	
	The above example shows that $\chi_\ell(G) \leq \chi_c(G)$ for all graphs $G$. As in the case of ordinary vs. list chromatic number, this inequality can be strict. For instance, if $C_{2n}$ is the cycle of length $2n$, then $\chi_\ell(C_{2n}) = 2$, while $\chi_c(C_{2n}) = 3$. Nevertheless, several known upper bounds for list coloring can be transferred to the correspondence coloring setting. For example, it is not hard to show that $\chi_c(G) \leq \Delta+1$ for any graph $G$ with maximum degree $\Delta$. Dvo\v{r}\'{a}k and Postle observed in \cite{DP} that $\chi_c(G) \leq 5$ if $G$ is planar, and $\chi_c(G) \leq 3$ if $G$ is planar and has girth at least $5$; these bounds are the analogs of Thomassen's results about list colorings~\cite{Th1}, \cite{Th2}.
	
	Our first result is an analog of Theorem~\ref{theo:Alon} for correspondence chromatic number.
	
	\begin{theo}\label{theo:lowerbound}
		Let $G$ be a graph with average degree $d \geq 2e$. Then
		$$
			\chi_c(G) \geq \frac{d/2}{\ln(d/2)}.
		$$
	\end{theo}
	
	Theorem~\ref{theo:lowerbound} shows that the correspondence chromatic number of a graph grows with the average degree much faster than the list chromatic number and, in fact, is only a logarithmic factor away from the trivial upper bound $\chi_c(G) \leq d+1$ for $d$-regular $G$.
	
	A celebrated theorem of Johansson~\cite{Johansson} asserts that for any triangle-free graph $G$ with maximum degree~$\Delta$, $\chi_\ell(G) = O(\Delta/\ln\Delta)$. Our next result shows that the same upper bound holds for correspondence coloring as well.
	
	\begin{theo}\label{theo:upperbound}
		There exists a positive constant $C$ such that for any triangle-free graph $G$ with maximum degree $\Delta$,
		$$
			\chi_c(G) \leq C \frac{\Delta}{\ln\Delta}.
		$$
	\end{theo}
	
	Combining Theorems~\ref{theo:lowerbound} and \ref{theo:upperbound}, we immediately obtain the following corollary.
	
	\begin{corl}
		There exist positive constants $c$ and $C$ such that for any $d$-regular triangle-free graph $G$, we have
		$$
			c \frac{d}{\ln d} \leq \chi_c(G) \leq C \frac{d}{\ln d}.
		$$
	\end{corl}
	
	In other words, the degree of a regular triangle-free graph determines its correspondence chromatic number up to a constant factor.
	
	The rest of the paper is organized as follows. In Section~\ref{sec:lowerbound} we prove Theorem~\ref{theo:lowerbound}. The proof is short and only uses the first moment method. Theorem~\ref{theo:upperbound} is proved in Section~\ref{sec:upperbound}. Most of the proof is done via adjusting Johansson's proof of the analogous result for ordinary graph colorings to the setting of correspondence colorings. However, some technical features of Johansson's proof have to be modified in order to work for correspondence colorings.
	
	\section{Proof of Theorem~\ref{theo:lowerbound}}\label{sec:lowerbound}
	
	In this section we prove Theorem~\ref{theo:lowerbound}. Suppose that
	$$
		k \leq \frac{d/2}{\ln(d/2)}.
	$$
	Let $\{L(v)\}_{v \in V(G)}$ be a collection of pairwise disjoint sets, each of size $k$. Randomly construct a graph $H$ with vertex set $\bigcup_{v \in V(G)} L(v)$ as follows: For each $uv \in E(G)$, connect $L(u)$ and $L(v)$ by a perfect matching chosen independently and uniformly at random. By construction, $(L, H)$ is a cover of $G$.
		
	Consider any set $I \subseteq \bigcup_{v \in V(G)} L(v)$ such that $|I \cap L(v)| = 1$ for all $v \in V(G)$. If $uv \in E(G)$, then the probability that the only vertex in $I\cap L(v)$ and the only vertex in $I\cap L(u)$ are nonadjacent in $H$ is exactly $1 - 1/k$. Therefore, the probability that $I$ is a proper $(L,H)$-coloring of $G$ is exactly $(1 - 1/k)^{|E(G)|} \leq e^{-|E(G)|/k}$. Thus, the probability that there exists at least one $(L,H)$-coloring is at most
	$$
		k^{|V(G)|} \cdot e^{-|E(G)|/k} = e^{|V(G)|\ln k - |E(G)|/k}.
	$$
	We claim that it is less than $1$. Indeed, it is enough to show that
	$$
		k\ln k < \frac{|E(G)|}{|V(G)|} = d/2.
	$$
	But
	$$
		k \ln k < \frac{d/2}{\ln (d/2)} \cdot \ln (d/2) = d/2,
	$$
	as desired.

	\section{Proof of Theorem~\ref{theo:upperbound}}\label{sec:upperbound}
	
	We will prove Theorem~\ref{theo:upperbound} in the following explicit form.
	
	\begin{theobis}{theo:upperbound}\label{theo:upperbound1}
		There exists a constant $\Delta_0$ such that for any triangle-free graph $G$ with maximum degree at most $\Delta \geq \Delta_0$,
		$$
			\chi_c(G) \leq \left\lceil\frac{120 \Delta}{\ln\Delta}\right\rceil.
		$$
	\end{theobis}
	
	 The proof of Theorem~\ref{theo:upperbound1} follows closely Johansson's proof of the analogous result for ordinary graph colorings, with only a few minor adjustments. Here we use the version of that proof given in Chapter 13 of~\cite{MR}. The only step that is significantly different is the proof of Lemma~\ref{lemma:degconc}, where the approach of \cite{MR} cannot be directly adapted for correspondence colorings.
	
	\subsection{Outline of the proof}
	
	From now on we will be working under the assumption that $\Delta$ is large, i.e., all lemmas are true only for $\Delta \geq \Delta_0$ for some constant $\Delta_0$.
	
	Our goal is to reduce Theorem~\ref{theo:upperbound1} to the following lemma.
	
	\begin{defn}
		Suppose that $(L,H)$ is a cover of $G$ and $p\colon V(H) \to \mathbb{R}_{\geq 0}$ is an assignment of nonnegative real numbers to the vertices of $H$. For each $v \in V(G)$, define
		$$
			p(v) \coloneqq \sum_{x \in L(v)} p(x)
		$$
		and for each $uv \in E(G)$, define
		$$
			p(uv) \coloneqq \sum_{xy \in E(L(u), L(v))} p(x)p(y).
		$$
	\end{defn}
	
	\begin{lemma}[Probabilistic Coloring Lemma]\label{lemma:final}
		Let $(L, H)$ be a cover of $G$ and let $p\colon V(H) \to \mathbb{R}_{\geq 0}$ be an assignment of nonnegative real numbers to the vertices of $H$. Suppose that there exists a positive constant $\delta$ such that
		\begin{enumerate}
			\item for each $v \in V(G)$, $p(v) \geq \delta$;
			\item for each $v \in V(G)$ and $x \in L(v)$, $p(x) \leq \delta/2$; and
			\item for each $v \in V(G)$, $\sum_{u \in N_G(v)} p(uv) \leq \delta^2/4$.			
		\end{enumerate}
		Then $G$ is $(L, H)$-colorable. 
	\end{lemma}
	
	We prove Lemma~\ref{lemma:final} in Subsection~\ref{subsec:final}.
	
	For technical reasons, we want to discard the colors $x$ for which $p(x)$ is too large. To do this, we fix a parameter $\hat{p}$ and introduce the following definitions.
	
	\begin{defn}[Moderate colors; moderate colorings; nice tuples]
		Suppose that $(L, H)$ is a cover of $G$ and $p\colon V(H) \to [0;\hat{p}]$. Call a color $x \in V(H)$ \emph{$p$-moderate} if $p(x) \in (0; \hat{p})$. If $M \subseteq V(H)$ is the set of all moderate colors, then for each $v \in V(G)$, define
		$$
			p_m(v) \coloneqq \sum_{x \in L(v) \cap M} p(x)
		$$
		and for each $uv \in E(G)$, define
		$$
			p_m(uv) \coloneqq \sum_{\substack{xy \in E(L(u), L(v));\\ x,y \in M}} p(x)p(y).
		$$
		Call an $(L,H)$-coloring $I$ \emph{$p$-moderate} if $I \subseteq M$. Finally, the tuple $(G, L, H, p)$ is \emph{nice} if there exists a positive constant $\delta$ such that
		\begin{enumerate}
			\item for each $v \in V(G)$, $p_m(v) \geq \delta$;
			\item for each $v \in V(G)$ and $x \in L(v) \cap M$, $p(x) \leq \delta/2$; and
			\item for each $v \in V(G)$, $\sum_{u \in N_G(v)} p_m(uv) \leq \delta^2/4$.			
		\end{enumerate}
	\end{defn}
	
	Lemma~\ref{lemma:final} immediately implies that if the tuple $(G, L, H, p)$ is nice, then $G$ admits a $p$-moderate $(L, H)$-coloring.
	
	\begin{defn}[Reducts]
		Let $(L, H)$ be a cover of $G$ and $p\colon V(H) \to [0;\hat{p}]$. A tuple $(G', L', H', p')$, where $G'$ is an induced subgraph of $G$, $L'(v) = L(v)$ for all $v \in V(G')$, $H'$ is the subgraph of $H$ induced by the set of vertices~$\bigcup_{v \in V(G')} L'(v)$, and $p'\colon V(H') \to [0; \hat{p}]$, is a \emph{reduct} of $(G, L, H, p)$ if every $p'$-moderate $(L',H')$-coloring of $G'$ extends to a $p$-moderate $(L,H)$-coloring of $G$.
	\end{defn}
	
	Reducts are useful because if $(G, L, H, p)$ has a nice reduct, then $G$ admits a $p$-moderate $(L,H)$-coloring.
	
	We need to introduce one more parameter, namely the \emph{entropy} of a vertex of $G$.
	
	\begin{defn}[Entropy]
		Suppose that $(L, H)$ is a cover of $G$ and $p\colon V(H) \to \mathbb{R}_{\geq 0}$. For each $v \in V(G)$, define the \emph{entropy} of $v$ to be
		$$
			Q(v) \coloneqq \sum_{x \in L(v)} p(x) \ln \frac{1}{p(x)},
		$$
		where we take $x \ln(1/x) = 0$ for $x = 0$.
	\end{defn}
	
	\begin{remk}
		To simplify notation, when talking about a tuple denoted $(G', L', H', p')$ (or $(G'', L'', H'', p'')$), we will denote the corresponding entropy by $Q'$ (resp. $Q''$).
	\end{remk}
	
	We will use the entropy of $v$ to control the difference between $p(v)$ and $p_m(v)$. Namely, we have the following:
	
	\begin{lemma}[Using the entropy]\label{lemma:entropy}
		Suppose that $G$ is a triangle-free graph with maximum degree at most~$\Delta$, $(L, H)$ is a cover of $G$ such that $|L(v)| = \left\lceil 120\Delta/\ln \Delta \right\rceil \eqqcolon k$ for all $v \in V(G)$, and $p\colon V(H) \to [0; \hat{p}]$, where $\hat{p} = \Delta^{-11/12}$. Moreover, assume that
		\begin{enumerate}
			\item for each $v \in V(G)$, $|p(v) - 1| \leq \Delta^{-1/10}$;
			\item for each $v \in V(G)$, $Q(v) \geq \ln k - \ln\Delta/40$;
			and
			\item for each $x \in V(H)$, either $p(x) = 0$ or $p(x) \geq 1/k$.
		\end{enumerate}
		Then for each $v \in V(G)$,
		$$
			p_m(v) \geq \frac{7}{10} - o(1).
		$$
	\end{lemma}
	
	Lemma~\ref{lemma:entropy} is proved in Subsection~\ref{subsec:entropy}.
	
	The next lemma is the core of the proof. Informally, it asserts that if $p(v) \approx 1$ for all $v \in V(G)$, $p(uv) = O(1/k)$ for all $uv \in E(G)$, and $Q(v) = \Omega(\ln \Delta)$ for all $v \in V(G)$, then there is a reduct $(G', L', H', p')$ of $(G, L, H, p)$ such that $p'(v) \approx p(v)$ for all $v \in V(G)$, $p'(uv) \approx p(uv)$ for all $uv \in E(G)$, and $Q'(v) \approx Q(v)$ for all $v \in V(G)$, but $\deg_{G'}(v)$ is \emph{much less} than $\deg_G(v)$ for all $v \in V(G)$. 
	
	\begin{lemma}[Good reducts]\label{lemma:reduct}
		Suppose that $G$ is a triangle-free graph with maximum degree at most $\Delta$, $(L, H)$ is a cover of $G$ such that $|L(v)| = \left\lceil 120\Delta/\ln \Delta\right\rceil \eqqcolon k$ for all $v \in V(G)$, and $p\colon V(H) \to [0; \hat{p}]$, where $\hat{p} \coloneqq \Delta^{-11/12}$. Moreover, assume that
		\begin{enumerate}
			\item for each $v \in V(G)$, $|p(v) - 1| \leq \Delta^{-1/10}$;
			\item for each $uv \in E(G)$, $p(uv) \leq \sqrt{2} k^{-1}$;
			\item for each $v \in V(G)$, $Q(v) \geq \ln k - \ln\Delta/40$; and
			\item for each $x \in V(H)$, either $p(x) = 0$ or $p(x) \geq 1/k$.
		\end{enumerate}
		Then there exists a reduct $(G', L', H', p')$ of $(G, L, H, p)$ such that
		\begin{enumerate}
			\item for each $v \in V(G')$, $|p'(v) - p(v)| \leq \Delta^{-1/6}$;
			\item for each $uv \in E(G')$, $p'(uv)\leq p(uv) + k^{-1}\Delta^{-1/3}$;
			\item for each $v \in V(G')$, $Q'(v) \geq Q(v) - 2\deg_G(v)/(k\ln\Delta) - \ln\Delta \cdot \Delta^{-1/6}$;
			\item for each $v \in V(G')$, $\deg_{G'}(v) \leq \deg_G(v) \cdot(1-2/(3\ln\Delta)) + \Delta^{2/3}$; and
			\item for each $x \in V(H')$, either $p'(x) = 0$ or $p'(x) \geq 1/k$.
		\end{enumerate}
	\end{lemma}
	
	We present the proof of Lemma~\ref{lemma:reduct} in Subsection~\ref{subsec:reduct}.
	
	Note that if $(G', L', H', p')$ is a reduct of $(G, L, H, p)$, and $(G'', L'', H'', p'')$ is a reduct of $(G', L', H', p')$, then $(G'', L'', H'', p'')$ is a reduct of $(G, L, H, p)$. Thus, Lemma~\ref{lemma:reduct} can be applied repeatedly to produce reducts with smaller and smaller degrees, until we finally obtain a nice tuple. The details of this strategy are worked out in Subsection~\ref{subsec:last}.
	
	\subsection{Proof of Lemma~\ref{lemma:final}}\label{subsec:final}
	
	We obtain Lemma~\ref{lemma:final} as an immediate corollary of the following result, proved in~\cite{Bernshteyn}.
	
	Let $U_1$, \ldots, $U_n$ be a collection of pairwise disjoint nonempty finite sets. A \emph{choice function} $F$ is a subset of $\bigcup_{i=1}^n U_i$ such that for all $1 \leq i \leq n$, $|F \cap U_i| = 1$. A \emph{partial choice function} $P$ is a subset of $\bigcup_{i=1}^n U_i$ such that for all $1 \leq i \leq n$, $|P \cap U_i| \leq 1$. For a partial choice function $P$, let
	$$
		\operatorname{dom}(P) \coloneqq \{i \,:\, P \cap U_i \neq \emptyset\}. 
	$$
	Thus, a choice function $F$ is a partial choice function with $\operatorname{dom}(F) = \{1, \ldots, n\}$.
	
	Let $F$ be a choice function and let $P$ be a partial choice function. We say that $P$ \emph{occurs} in $F$ if $P \subseteq F$, and we say that $F$ \emph{avoids} $P$ if $P$ does not occur in $F$.
	
	A \emph{multichoice function} $M$ is any subset of $\bigcup_{i=1}^n U_i$. Again, we say that a partial choice function $P$ \emph{occurs} in a multichoice function $M$ if $P \subseteq M$. Suppose that we are given a family $P_1$, \ldots, $P_m$ of nonempty ``forbidden'' partial choice functions. For a multichoice function $M$, the \emph{$i^\text{th}$ defect} of $M$ (notation: $\operatorname{def}_i(M)$) is the number of indices $j$ such that $i \in \operatorname{dom}(P_j)$ and $P_j$ occurs in $M$.
	
	\begin{theo}[\cite{Bernshteyn}, Theorem~20]\label{theo:choice}
		Let $U_1$, \ldots, $U_n$ be a collection of pairwise disjoint nonempty finite sets and let $P_1$, \ldots, $P_m$ be a family of nonempty partial choice functions. Let $M_i \subseteq U_i$ be a random subset of $U_i$ for each $1 \leq i \leq n$. Suppose that the variables $M_i$ are mutually independent and let $M \coloneqq \bigcup_{i=1}^n M_i$. If
		$$
			\mathbf{E}(|M_i|) \geq 1 + \mathbf{E} (\operatorname{def}_i(M))
		$$
		for all $1 \leq i \leq n$, then there exists a choice function $F$ that avoids all of $P_1$, \ldots, $P_m$.
	\end{theo}
	
	Now let us state Lemma~\ref{lemma:final} again.
	
	\begin{lemmacopy}{lemma:final}\label{lemma:final1}
		Let $(L, H)$ be a cover of $G$ and let $p\colon V(H) \to \mathbb{R}_{\geq 0}$ be an assignment of nonnegative real numbers to the vertices of $H$. Suppose that there exists a positive constant $\delta$ such that
		\begin{enumerate}
			\item for each $v \in V(G)$, $p(v) \geq \delta$;
			\item for each $v \in V(G)$ and $x \in L(v)$, $p(x) \leq \delta/2$; and
			\item for each $v \in V(G)$, $\sum_{u \in N_G(v)} p(uv) \leq \delta^2/4$.			
		\end{enumerate}
		Then $G$ is $(L, H)$-colorable. 
	\end{lemmacopy}
	\begin{proof}
		For each $xy \in E(H)$, let $P_{xy} \coloneqq \{x, y\}$. Note that such $P_{xy}$ is a partial choice function with respect to the collection $\{L(v)\}_{v \in V}$ of pairwise disjoint sets. Moreover, a choice function that avoids all of $\{P_{xy}\}_{xy \in E(H)}$ is an $(L,H)$-coloring of $G$.
		
		Construct a random set $M \subseteq V(H)$ by including a vertex $x \in V(H)$ in $M$ with probability $2p(x)/\delta$, making the choices independently for all vertices. We have
		$$
			\mathbf{E}(|M\cap L(v)|) = \frac{2p(v)}{\delta}
		$$
		and
		$$
			\mathbf{E}(\operatorname{def}_v(M)) = \frac{4}{\delta^2}\sum_{u \in N_G(v)} p(uv).
		$$
		Therefore, Theorem~\ref{theo:choice} implies that if for all $v \in V(G)$, we have
		$$
			\frac{2p(v)}{\delta} \geq 1 + \frac{4}{\delta^2}\sum_{u \in N_G(v)} p(uv),
		$$
		then there exists a choice function avoiding all of $\{P_{xy}\}_{xy \in E(G)}$, i.e., $G$ is $(L,H)$-colorable. But
		$$
			\frac{2p(v)}{\delta} \geq 2,
		$$
		while
		$$
			\frac{4}{\delta^2}\sum_{u \in N_G(v)} p(uv) \leq 1,
		$$
		so we are done.
	\end{proof}
	\begin{remk}
		A result similar to Lemma~\ref{lemma:final1} can also be obtained using the Lov\'{a}sz Local Lemma.
	\end{remk}
	\begin{remk}
		An argument analogous to the proof of Theorem~\ref{theo:choice} given in~\cite{Bernshteyn} can be used to show that the second condition in Lemma~\ref{lemma:final1} is, in fact, not necessary.
	\end{remk}	
	
	\subsection{Proof of Lemma~\ref{lemma:entropy}}\label{subsec:entropy}
	
	Let us recall the statement of the lemma.
	
	\begin{lemmacopy}{lemma:entropy}
		Suppose that $G$ is a triangle-free graph with maximum degree at most~$\Delta$, $(L, H)$ is a cover of $G$ such that $|L(v)| = \left\lceil 120\Delta/\ln \Delta \right\rceil \eqqcolon k$ for all $v \in V(G)$, and $p\colon V(H) \to [0; \hat{p}]$, where $\hat{p} = \Delta^{-11/12}$. Moreover, assume that
		\begin{enumerate}
			\item for each $v \in V(G)$, $|p(v) - 1| \leq \Delta^{-1/10}$;
			\item for each $v \in V(G)$, $Q(v) \geq \ln k - \ln\Delta/40$;
			and
			\item for each $x \in V(H)$, either $p(x) = 0$ or $p(x) \geq 1/k$.
		\end{enumerate}
		Then for each $v \in V(G)$,
		$$
		p_m(v) \geq \frac{7}{10} - o(1).
		$$
	\end{lemmacopy}
	\begin{proof}
		Let
		$$
			B \coloneqq \{x \in V(H)\,:\, p(x) = \hat{p}\}.
		$$
		Note that
		$$
			p_m(v) = p(v) - |L(v) \cap B| \cdot \hat{p},
		$$
		so we need to get an upper bound on $|L(v) \cap B|\cdot \hat{p}$.
		 
		We have
		$$
			\ln k = \sum_{x \in L(v)} p(x)\ln k + (1 - p(v))\ln k \geq \sum_{x \in L(v)}p(x) \ln k - \Delta^{-1/10}\ln k.
		$$
		Therefore,
		$$
			\frac{\ln\Delta}{40} \geq \ln k - Q(v) \geq \sum_{x \in L(v)} p(x) \ln(k\cdot p(x)) - o(\ln\Delta).
		$$
		Note that all the terms in the latter sum are nonnegative. Each $x \in L(v) \cap B$ contributes to the sum the quantity
		$$
			\hat{p} \ln(k \cdot \hat{p}) = \frac{1}{12} \hat{p} \ln\Delta - o(\ln \Delta).
		$$
		Thus,
		$$
			\frac{\ln\Delta}{40} \geq |L(v) \cap B| \cdot \hat{p} \cdot \frac{1}{12}\ln\Delta - o(\ln\Delta).
		$$
		Hence,
		$$
			|L(v) \cap B| \cdot \hat{p} \leq \frac{3}{10} + o(1).
		$$
		This yields the lower bound
		$$
			p_m(v) \geq p(v) - \frac{3}{10} - o(1) \geq \frac{7}{10} - o(1),
		$$
		as desired.
	\end{proof}

	\subsection{Proof of Lemma~\ref{lemma:reduct}}\label{subsec:reduct}
	
	\subsubsection{The algorithm}
	
	Fix a positive parameter $\alpha$. Given a tuple $(G, L, H, p)$, the following randomized procedure outputs its reduct $(G', L', H', p')$.
	
	Let
	$$
		B \coloneqq \left\{x \in V(H) \,:\, p(x) = \hat{p}\right\}.
	$$
	For $x \in V(H) \setminus B$, let
	$$
		K(x) \coloneqq \prod_{y \in N_H(x) \setminus B} \left(1-\alpha p(y)\right).
	$$
	Let
	$$
		q(x) \coloneqq \frac{p(x)/\hat{p} - K(x)}{1 - K(x)}.
	$$
	
	Generate a random subset $S \subseteq V(H) \setminus B$ by including a vertex $x \in V(H) \setminus B$ in $S$ with probability $\alpha p(x)$, making the choices independently for all vertices.
	
	Define the function $p'\colon V(H) \to [0; \hat{p}]$ as follows.
	\begin{itemize}
		\item For each $x \in B$, let $p'(x) \coloneqq \hat{p}$.
		\item For each $x \not\in B$ such that $N_H(x) \cap S = \emptyset$, let $p'(x) \coloneqq \min\{p(x)/K(x), \hat{p}\}$.
		\item For each $x \not \in B$ such that $N_H(x) \cap S \neq \emptyset$ and $p(x)/K(x) \leq \hat{p}$, let $p'(x) \coloneqq 0$.
		\item For each $x \not \in B$ such that $N_H(x) \cap S \neq \emptyset$ and $p(x)/K(x) > \hat{p}$, let $p'(x) \coloneqq \hat{p}$ with probability $q(x)$ and $p'(x) \coloneqq 0$ otherwise.
	\end{itemize}
	
	Let
	$$
		W \coloneqq \left\{x \in V(H) \,:\, x \in S \text{ and } N_H(x) \cap S = \emptyset\right\};
	$$
	$$
		A \coloneqq \left\{v \in V(G) \,:\, L(v) \cap S \subseteq W \text{ and } L(v) \cap S \neq \emptyset\right\}.
	$$
	Finally, let
	$$
	G' \coloneqq G - A
	$$
	and define $L'$ and $H'$ accordingly.
	
	\begin{lemma}\label{lemma:correct}
		The above procedure always outputs a tuple $(G', L', H', p')$ that is a reduct of $(G, L, H, p)$. Moreover, for each $x \in V(H')$, either $p'(x) = 0$ or $p'(x) \geq p(x)$.
	\end{lemma}
	\begin{proof}
		The second part of the statement follows immediately from the way $p'$ is defined. To prove the first part, we claim that if $I' \subseteq V(H')$ is a $p'$-moderate $(L', H')$-coloring of $G'$, then $$I \coloneqq I' \cup \left(\bigcup_{v \in A} L(v) \cap S\right)$$ is a $p$-moderate $(L,H)$-coloring of $G$. Note that for each $v \in V(G)$, $I \cap L(v) \neq \emptyset$. Indeed, if $v \in V(G')$, then $I \cap L(v) = I'\cap L(v) \neq \emptyset$. If, on the other hand, $v \in A$, then $I \cap L(v) = L(v) \cap S \neq \emptyset$. Moreover, $I$ is independent. Indeed, $W$ is an independent set since $W \subseteq S$ but $N_H(x) \cap S = \emptyset$ for all $x \in W$. Therefore, $\bigcup_{v \in A} L(v) \cap S \subseteq W$ is independent as well. Since $I'$ is independent by definition, if there is an edge between two vertices in $I$, then it connects a vertex $x \in \bigcup_{v \in A} L(v) \cap S$ and a vertex $y \in I'$. Since $x \in N_H(y) \cap S$, $p'(y) \in \{0, \hat{p}\}$, which is a contradiction because $I'$ is $p'$-moderate.
		
		Therefore, $I$ is an $(L,H)$-coloring. It remains to verify that $I$ is $p$-moderate. Note that, for all $x \in V(H')$, if $p(x) \in \{0, \hat{p}\}$, then $p'(x) \in \{0, \hat{p}\}$ as well. Thus, since $I'$ is $p'$-moderate, it is enough to establish that $p(x) \in (0; \hat{p})$ for all $x \in \bigcup_{v \in A} L(v) \cap S$. Indeed, if $x \in S$, then $p(x) > 0$, for otherwise $x$ would not have been chosen, and $x \not \in B$, so $p(x) < \hat{p}$. 
	\end{proof}
	
	\subsubsection{Analysis of the algorithm}
	
	\paragraph{Probabilistic tools and general remarks}
	
	For the rest of the proof we assume that $G$ is a triangle-free graph with maximum degree at most $\Delta$ and $|L(v)| = \left\lceil 120\Delta/\ln\Delta\right\rceil \eqqcolon k$ for all $v \in V(G)$. We also set $\hat{p} \coloneqq \Delta^{-11/12}$ and $\alpha \coloneqq 1/\ln\Delta$.
	
	For each $x \in V(H)$, let $R(x)$ be the event that $N_H(x) \cap S \neq \emptyset$. Note that the value of $p'(x)$ is determined by the outcome of the event $R(x)$. Also observe that
	\begin{itemize}
		\item $N_H(x_1) \cap N_H(x_2) = \emptyset$ for any two distinct $x_1$, $x_2 \in L(v)$, where $v \in V(G)$;
		\item $N_H(x) \cap N_H(y) = \emptyset$ for any $x \in L(u)$, $y \in L(v)$, where $uv \in E(G)$ (this follows from the fact that $G$ is triangle-free).
	\end{itemize}
	Therefore, the following sets of random events are mutually independent:
	\begin{itemize}
		\item $\{R(x)\,:\, x \in L(v)\}$ for each $v \in V(G)$;
		\item $\{R(x)\,:\, x \in L(u)\} \cup \{R(y)\,:\, y \in L(v)\}$ for each $uv \in E(G)$.
	\end{itemize}
	
	Note that we have defined the value $p'(x)$ for all $x \in V(H)$ (and not only for $x \in V(H')$). Therefore, we can extend the definitions of $p'(v)$, $p'(uv)$, and $Q'(v)$ to all $v \in V(G)$ and $uv \in E(G)$.
	
	We will use the following standard results from probability theory.
	
	\begin{theo}[The Chernoff Bound, \cite{MR}, page 43]
		Suppose that $X_1$, \ldots, $X_n$ are independent random variables, each equal to $1$ with probability $p$ and $0$ otherwise. Let $X \coloneqq \sum_{i=1}^n X_i$. Then for any $0 \leq t \leq np$,
		$$
			\Pr(|X - np| > t) < 2e^{-t^2/(3np)}.
		$$
	\end{theo}
	
	\begin{theo}[Simple Concentration Bound, \cite{MR}, page 79]
		Let $X$ be a random variable determined by $n$ independent trials $T_1$, \ldots, $T_n$ and such that changing the outcome of any one trial can affect $X$ by at most~$c$. Then
		$$
		\Pr(|X - \mathbf{E}(X)| > t) \leq 2 e^{-t^2/(2c^2 n)}.
		$$
	\end{theo}
	
	\begin{theo}[Talagrand's Inequality, \cite{MR}, page 81]
		Let $X$ be a nonnegative random variable, not identically~$0$, which is determined by $n$ independent trials $T_1$, \ldots, $T_n$, and satisfying the following for some $c$, $r > 0$:
		\begin{enumerate}
			\item changing the outcome of any one trial can affect $X$ by at most $c$; and
			\item for any $s$, if $X \geq s$, then there is a set of at most $rs$ trials whose outcomes certify that $X \geq s$.
		\end{enumerate}
		Then for any $0 \leq t \leq \mathbf{E}(X)$,
		$$
			\Pr(|X - \mathbf{E}(X)| > t + 60c\sqrt{r\mathbf{E}(X)}) \leq 4e^{-t^2/(8c^2 r \mathbf{E}(X))}.
		$$
	\end{theo}
	
	\begin{theo}[Symmetric Lov\'{a}sz Local Lemma, \cite{AS}, Corollary 5.1.2]
		Suppose that $\mathcal{A}$ is a finite set of random events such that
		\begin{enumerate}
			\item for each $A \in \mathcal{A}$, $\Pr(A) \leq p$; and
			\item for each $A \in \mathcal{A}$, there is a set $\Gamma(A) \subseteq \mathcal{A}\setminus \{A\}$ of size at most $d$ such that $A$ is mutually independent from the collection of events $\mathcal{A} \setminus (\Gamma(A) \cup \{A\})$.
		\end{enumerate}
		If $ep(d+1) \leq 1$, then with positive probability none of the events in $\mathcal{A}$ happen.
	\end{theo}

	\paragraph{Estimating $p'(v)$ and $p'(uv)$}
	
	In this paragraph we show that for each $v \in V(G)$, the value of the random variable $p'(v)$ is highly concentrated around $p(v)$; and for each $uv \in E(G)$, the value of the random variable $p'(uv)$ is highly concentrated around $p(uv)$.
	
	\begin{lemma}\label{lemma:vertexexp}
		For each $v \in V(G)$,
		$$
		\mathbf{E}(p'(v)) = p(v).
		$$
	\end{lemma}
	\begin{proof}
		It is enough to show that for all $x \in V(H)$,
		$$
			\mathbf{E}(p'(x)) = p(x).
		$$
		If $x \in B$, then $p'(x) = p(x) = \hat{p}$, so suppose that $x \not \in B$. There are two cases.
		
		{\sc Case 1:} $p(x)/K(x) \leq \hat{p}$. In this case,
		$$
			p'(x) = \begin{cases}
				p(x)/K(x) &\text{if $N_H(x) \cap S = \emptyset$};\\
				0 &\text{otherwise}.
			\end{cases}
		$$
		Note that
		$$
			\Pr(N_H(x) \cap S = \emptyset) = \prod_{y \in N_H(x) \setminus B} (1 - \alpha p(y)) = K(x),
		$$
		so
		$$
			\mathbf{E}(p'(x)) = p(x)/K(x) \cdot \Pr(N_H(x) \cap S = \emptyset) = p(x),
		$$
		as desired.
		
		{\sc Case 2:} $p(x)/K(x) > \hat{p}$. In this case, $p'(x) = \hat{p}$ if $N_H(x) \cap S = \emptyset$, and if $N_H(x) \cap S \neq \emptyset$, then $p'(x) = \hat{p}$ with probability $q(x)$ and $p'(x) = 0$ otherwise. Therefore,
		$$
			\mathbf{E}(p'(x)) = \hat{p} \cdot K(x) + \hat{p} \cdot (1-K(x)) \cdot q(x) = p(x),
		$$
		as desired.
	\end{proof}
	
	\begin{lemma}\label{lemma:edgeexp}
		For each $uv \in E(G)$,
		$$
		\mathbf{E}(p'(uv)) = p(uv).
		$$
	\end{lemma}
	\begin{proof}
		It is enough to show that for all $xy \in E(L(u), L(v))$,
		$$
		\mathbf{E}(p'(x)p'(y)) = p(x)p(y).
		$$
		Since $G$ is triangle-free, the variables $p'(x)$ and $p'(y)$ are independent. Therefore,
		$$
		\mathbf{E}(p'(x)p'(y)) = \mathbf{E}(p'(x)) \cdot \mathbf{E}(p'(y)) = p(x) p(y),
		$$
		as desired.
	\end{proof}
	
	\begin{lemma}\label{lemma:vertexconc}
		For each $v \in V(G)$,
		$$
			\Pr(|p'(v) - p(v)| > \Delta^{-1/6}) \leq \Delta^{-5}.
		$$
	\end{lemma}
	\begin{proof}
		The desired result follows immediately via the Simple Concentration Bound applied to the series of trials $\{R(x)\,:\, x \in L(v)\}$.
	\end{proof}
	
	\begin{lemma}\label{lemma:edgeconc}
		For each $uv \in E(G)$,
		$$
		\Pr(|p'(uv) - p(uv)| > k^{-1}\Delta^{-1/3}) \leq \Delta^{-5}.
		$$
	\end{lemma}
	\begin{proof}
		The desired result follows immediately via the Simple Concentration Bound applied to the series of trials $\{R(x)\,:\, x \in L(u)\} \cup \{R(y)\,:\, y \in L(v)\}$.
	\end{proof}
	
	\paragraph{Estimating the entropy}
	
	In this paragraph we show that, with high probability, $Q'(v)$ is not much smaller than $Q(v)$.
	
	\begin{lemma}\label{lemma:entropyexp}
		Suppose that for all $uv \in E(G)$, $p(uv) \leq P$. Then for each $v \in V(G)$,
		$$
			\mathbf{E}(Q'(v)) \geq Q(v) - \frac{\sqrt{2}P}{\ln \Delta} \cdot \deg_G(v).
		$$
	\end{lemma}
	\begin{proof}
		Consider some $x \in L(v) \setminus B$. If $p(x)/K(x) \leq \hat{p}$, then we have
		$$
			\mathbf{E}\left(p'(x) \ln \frac{1}{p'(x)}\right) = p(x)/K(x) \cdot \ln \frac{K(x)}{p(x)} \cdot K(x) = p(x)\ln \frac{1}{p(x)} - p(x) \ln \frac{1}{K(x)}.
		$$
		If, on the other hand, $p(x)/K(x) > \hat{p}$, then
		\begin{align*}
			\mathbf{E}\left(p'(x) \ln \frac{1}{p'(x)}\right) &= \hat{p} \ln \frac{1}{\hat{p}} \cdot K(x) + \hat{p} \ln \frac{1}{\hat{p}} \cdot(1 - K(x)) \cdot q(x) \\
			&= p(x) \ln \frac{1}{p(x)} - p(x) \ln \frac{\hat{p}}{p(x)} \geq p(x)\ln \frac{1}{p(x)} - p(x) \ln \frac{1}{K(x)}.
		\end{align*}
		Since for $x \in B$, $p'(x) = p(x) = \hat{p}$, we get
		\begin{align*}
			\mathbf{E}(Q'(v)) &\geq Q(v) - \sum_{x \in L(v) \setminus B} p(x) \ln \frac{1}{K(x)}\\
			&\geq Q(v) - \sum_{x \in L(v)} p(x) \ln \frac{1}{K(x)}.
		\end{align*}
		
		Note that for sufficiently small $x$,
		$$
			\ln\frac{1}{1-x} \leq x(1+ x).
		$$
		Thus,
		\begin{align*}
			\ln \frac{1}{K(x)} &= \ln \frac{1}{\prod_{y \in N_H(x) \setminus B} (1 - \alpha p(y))} \\
			&\leq \sum_{y \in N_H(x) \setminus B} \alpha p(y)(1 + \alpha p(y))\\
			&\leq \sum_{y \in N_H(x)} \alpha p(y)(1 + \alpha p(y)).
		\end{align*}
		Therefore,
		\begin{align*}
			\mathbf{E}(Q'(v)) &\geq Q(v) - \sum_{x \in L(v)} \sum_{y \in N_H(x)} \alpha p(x) p(y) (1 + \alpha p(y))\\
			&\geq Q(v) - \alpha (1 + \alpha \hat{p}) \sum_{u \in N_G(v)} p(uv) \\
			& \geq Q(v) - \frac{\sqrt{2}P}{\ln \Delta}\cdot\deg_G(v),
		\end{align*}
		as desired.
	\end{proof}
	
	\begin{lemma}\label{lemma:entropyconc}
		For each $v \in V(G)$,
		$$
		\Pr(|Q'(v) - \mathbf{E}(Q'(v))| > \ln \Delta \cdot \Delta^{-1/6}) \leq \Delta^{-5}.
		$$
	\end{lemma}
	\begin{proof}
		The desired result follows immediately via the Simple Concentration Bound applied to the series of trials $\{R(x)\,:\, x \in L(v)\}$.
	\end{proof}
	
	\paragraph{Estimating the degrees}
	
	In this paragraph we show that, with high probability, the vertices of $G'$ have smaller degrees than the vertices of $G$. For each $v \in V(G)$, define the following random variable:
	$$
		d'(v) \coloneqq |N_G(v) \cap V(G')|.
	$$
	Note that for each $v \in V(G')$, $d'(v) = \deg_{G'}(v)$.
	
	\begin{lemma}\label{lemma:degexp}
		Suppose that for all $v \in V(G)$,
		$$
			p_1 \leq p_m(v) \leq p_2,
		$$
		and for all $uv \in E(G)$, $p(uv) \leq P$.
		Then for each $v \in V(G)$,
		$$
			\mathbf{E}(d'(v)) \leq \deg_G(v) \cdot \left(1 - \alpha p_1 + \alpha^2 (p_2^2 +P \Delta)\right).
		$$
	\end{lemma}
	\begin{proof}
		It is enough to show that for each $v \in V(G)$,
		$$
			\Pr(v \in V(G')) \leq 1 - \alpha p_1 + \alpha^2 (p_2^2 +P \Delta),
		$$
		i.e.,
		$$
			\Pr(v \in A) \geq \alpha p_1 - \alpha^2 (p_2^2 +P \Delta).
		$$
		We have
		\begin{align*}
			\Pr(L(v) \cap S \neq \emptyset) &\geq \sum_{x \in L(v) \setminus B} \alpha p(x) - \sum_{\substack{x, y \in L(v)\setminus B,\\ x \neq y}} \alpha^2 p(x) p(y) \\
			&\geq \sum_{x \in L(v) \setminus B} \alpha p(x) - \left(\sum_{x \in L(v) \setminus B} \alpha p(x)\right)^2 \\
			&\geq \alpha p_1 - \alpha^2 p_2^2.
		\end{align*}
		Also,
		\begin{align*}
			\Pr(L(v) \cap S \not\subseteq W) & \leq \sum_{x \in L(v)} \sum_{y \in N_H(x)} \alpha^2 p(x) p(y) \\
			& = \alpha^2 \sum_{u \in N_G(v)} p(uv)\\
			& \leq \alpha^2 P \Delta.
		\end{align*}
		Thus,
		\begin{align*}
			\Pr(v \in A) \geq \Pr(L(v) \cap S \neq \emptyset) - \Pr(L(v) \cap S \not\subseteq W) \geq \alpha p_1 - \alpha^2 (p_2^2 +P \Delta),
		\end{align*}
		as desired.
	\end{proof}
	
	\begin{lemma}\label{lemma:degconc}
		For each $v \in V(G)$,
		$$
			\Pr(|d'(v) - \mathbf{E}(d'(v))| > \Delta^{2/3}) \leq \Delta^{-5}.
		$$
	\end{lemma}
	\begin{proof}
		Here our proof diverges from the proof of the analogous statement for ordinary colorings given in~\cite{MR}. Let
		$$
			N_1 \coloneqq \bigcup_{u \in N_G(v)} L(u)
		$$
		and
		$$
			N_2 \coloneqq \bigcup_{u \in N_G(v)} \bigcup_{w \in N_G(u)} L(w).
		$$
		Since $G$ is triangle-free, $N_1 \cap N_2 = \emptyset$. Note that to determine $d'(v)$, it is enough to specify which colors in $N_1 \cup N_2$ belong to $S$. Call a color $x\in N_2$ \emph{saturated} if $|N_H(x) \cap S| > \Delta^{1/10}$.
		
		\begin{claim}\label{claim:nosat}
			With probability at least $1 - \Delta^{-6}$, there are no saturated colors.
		\end{claim}
		\begin{claimproof}
			The Chernoff Bound implies that, for sufficiently large $\Delta$, the probability that a given color $x \in N_2$ is saturated is less that $\Delta^{-9}$. Since $|N_2| \leq \Delta^2 \cdot k \leq \Delta^3$, the claim follows.
		\end{claimproof}
		
		For a set $T \subseteq N_1$, let $\mathbf{E}_T$ denote the operator of conditional expectation given that $N_1 \cap S = T$.
		
		\begin{claim}\label{claim:condexp}
			Suppose that $T \subseteq N_1$ is chosen randomly according to the distribution of $N_1 \cap S$. Then with probability at least $1 -\Delta^{-6}$,
			$$
				|\mathbf{E}(d'(v)) - \mathbf{E}_T(d'(v))| \leq \frac{1}{2}\Delta^{2/3}.
			$$
		\end{claim}
		\begin{claimproof}
			Note that the following set of $\deg_G(v) \leq \Delta$ independent trials determines $T$:
			$$
				\{L(u) \cap S \,:\, u \in N_G(v)\}.
			$$
			Moreover, changing the outcome of any one trial can affect $\mathbf{E}_T(d'(v))$ by at most $1$. The result thus follows immediately from the Simple Concentration Bound.
		\end{claimproof}
		
		Note that to know which colors in $N_2$ are saturated, it is enough to know the set $N_1 \cap S$. Call a color $x \in N_2$ \emph{$T$-saturated} if it is saturated when $N_1 \cap S = T$.
		
		\begin{claim}\label{claim:cond}
			Suppose that $T \subseteq N_1$ is fixed and such that no color in $N_2$ is $T$-saturated. Then, given that $N_1 \cap S = T$, with probability at least $1 - \Delta^{-6}$,
			$$
				|d'(v) - \mathbf{E}_T(d'(v))| \leq \frac{1}{2}\Delta^{2/3}.
			$$
		\end{claim}
		\begin{claimproof}
			With $N_1 \cap S$ fixed, the value $d'(v)$ is determined by $|N_2|$ independent trials corresponding to the colors in $N_2$. Changing the outcome of any one trial can affect $d'(v)$ by at most $\Delta^{1/10}$ since no color in $N_2$ is saturated. Therefore, we can apply Talagrand's Inequality with $c = \Delta^{1/10}$, $r = 1$, and using the fact that $\mathbf{E}_T(d'(v)) \leq \Delta$.
		\end{claimproof}
		
		Putting together Claims~\ref{claim:nosat}, \ref{claim:condexp}, and \ref{claim:cond}, we get
		$$
			\Pr(|d'(v) - \mathbf{E}(d'(v))| > \Delta^{2/3}) \leq 3 \Delta^{-6} < \Delta^{-5},
		$$
		as desired.
	\end{proof}
	
	\subsubsection{Completing the proof of Lemma~\ref{lemma:reduct}}	
	
	Let us now state the lemma again.
	
	\begin{lemmacopy}{lemma:reduct}
		Suppose that $G$ is a triangle-free graph with maximum degree at most $\Delta$, $(L, H)$ is a cover of $G$ such that $|L(v)| = \left\lceil 120\Delta/\ln \Delta\right\rceil \eqqcolon k$ for all $v \in V(G)$, and $p\colon V(H) \to [0; \hat{p}]$, where $\hat{p} \coloneqq \Delta^{-11/12}$. Moreover, assume that
		\begin{enumerate}
			\item for each $v \in V(G)$, $|p(v) - 1| \leq \Delta^{-1/10}$;
			\item for each $uv \in E(G)$, $p(uv) \leq \sqrt{2} k^{-1}$;
			\item for each $v \in V(G)$, $Q(v) \geq \ln k - \ln\Delta/40$; and
			\item for each $x \in V(H)$, either $p(x) = 0$ or $p(x) \geq 1/k$.
		\end{enumerate}
		Then there exists a reduct $(G', L', H', p')$ of $(G, L, H, p)$ such that
		\begin{enumerate}
			\item for each $v \in V(G')$, $|p'(v) - p(v)| \leq \Delta^{-1/6}$;
			\item for each $uv \in E(G')$, $p'(uv)\leq p(uv) + k^{-1}\Delta^{-1/3}$;
			\item for each $v \in V(G')$, $Q'(v) \geq Q(v) - 2\deg_G(v)/(k\ln\Delta) - \ln\Delta \cdot \Delta^{-1/6}$;
			\item for each $v \in V(G')$, $\deg_{G'}(v) \leq \deg_G(v) \cdot(1-2/(3\ln\Delta)) + \Delta^{2/3}$; and
			\item for each $x \in V(H')$, either $p'(x) = 0$ or $p'(x) \geq 1/k$.
		\end{enumerate}
	\end{lemmacopy}
	\begin{proof}
		We just need to show that the tuple $(G', L', H', p')$, obtained using our procedure, satisfies properties 1--5 with positive probability. Note that, according to Lemma~\ref{lemma:correct}, property 5 is always satisfied. The other properties can be easily verified using the Lov\'{a}sz Local Lemma together with Lemmas~\ref{lemma:vertexexp}, \ref{lemma:vertexconc}, \ref{lemma:edgeexp}, \ref{lemma:edgeconc}, \ref{lemma:entropyexp}, \ref{lemma:entropyconc}, \ref{lemma:degexp}, and \ref{lemma:degconc}. We only need to compute the bound on $\mathbf{E}(d'(v))$ given by Lemma~\ref{lemma:degexp}. Due to Lemma~\ref{lemma:entropy}, we have
		$$
			\frac{7}{10}-o(1) \leq p_m(v) \leq 1 + o(1).
		$$
		Plugging these bounds into Lemma~\ref{lemma:degexp} gives
		$$
			\mathbf{E}(d'(v)) \leq \deg_G(v) \cdot\left(1 - \frac{7}{10\ln\Delta} + \frac{\sqrt{2}\Delta}{k(\ln\Delta)^2} + o\left(\frac{1}{\ln\Delta}\right)\right) < \deg_G(v)\cdot\left(1 - \frac{2}{3\ln\Delta}\right),
		$$
		as desired.
	\end{proof}

	\subsection{Finishing the proof of Theorem~\ref{theo:upperbound}}\label{subsec:last}
	
	Let $G$ be a triangle-free graph with maximum degree at most $\Delta$ and suppose that $(L, H)$ is a cover of $G$ such that $|L(v)| = \left\lceil 120\Delta/\ln \Delta\right\rceil \eqqcolon k$ for all $v \in V(G)$. Define $p(x) \coloneqq 1/k$ for all $x \in V(H)$ and let $\hat{p} \coloneqq \Delta^{-11/12}$. Observe that
	\begin{enumerate}
		\item for each $v \in V(G)$, $p(v) = 1$;
		\item for each $uv \in E(G)$, $p(uv) \leq k^{-1}$;
		\item for each $v \in V(G)$, $Q(v) = \ln k$.
	\end{enumerate}	
	
	Let $i^\ast$ be the least nonnegative integer such that
	$$
		\Delta \cdot\left(1-\frac{2}{3\ln\Delta}\right)^{i^\ast} + i^\ast\Delta^{2/3} \leq \left(\frac{3}{5}\right)^2 \cdot \frac{1}{4} \cdot \frac{1}{\sqrt{2}} \cdot k.
	$$
	Note that $i^\ast = O(\ln \Delta \ln\ln\Delta)$. After $i^\ast$ applications of Lemma~\ref{lemma:reduct} we obtain a reduct $(G'', L'', H'', p'')$ of $(G, L, H, p)$ such that
	\begin{enumerate}
		\item for each $v \in V(G'')$, $|p''(v) - 1| \leq i^\ast \Delta^{-1/6} < \Delta^{-1/10}$;
		\item for each $uv \in E(G'')$, $p''(uv)\leq k^{-1} + i^\ast k^{-1}\Delta^{-1/3} < \sqrt{2} k^{-1}$;
		\item for each $v \in V(G'')$,
		$$
			\deg_{G''}(v) \leq \Delta \cdot\left(1-\frac{2}{3\ln\Delta}\right)^{i^\ast} + i^\ast\Delta^{2/3} \leq \left(\frac{3}{5}\right)^2 \cdot \frac{1}{4} \cdot \frac{1}{\sqrt{2}} \cdot k;
		$$
		\item for each $v \in V(G'')$,
		\begin{align*}
			Q''(v) &\geq \ln k - \frac{2\Delta}{k\ln\Delta}\sum_{i = 0}^{i^\ast-1} \left(\left(1 - \frac{2}{3\ln\Delta}\right)^i + i \Delta^{-1/3}\right) - i^\ast \cdot\ln\Delta \cdot \Delta^{-1/6}\\
			&\geq \ln k - \frac{1}{40} \ln\Delta.
		\end{align*}
	\end{enumerate}
	
	We claim that the tuple $(G'', L'', H'', p'')$ is nice. Indeed, according to Lemma~\ref{lemma:entropy}, for all $v \in V(G'')$,
	$$
		p''_m(v) \geq \frac{7}{10} - o(1).
	$$
	On the other hand, for every $uv \in E(G'')$,
	$$
		p''_m(uv) \leq p''(uv) \leq \sqrt{2}k^{-1}.
	$$
	Therefore, for all $v \in V(G'')$,
	\begin{align*}
		\sum_{u \in N_{G''}(v)} p''_m(uv) 
		&\leq \left(\frac{3}{5}\right)^2 \cdot \frac{1}{4} \cdot \frac{1}{\sqrt{2}} \cdot k \cdot \sqrt{2}k^{-1} \\ &= \frac{1}{4} \cdot \left(\frac{3}{5}\right)^2 < \frac{1}{4} \left(\frac{7}{10}\right)^2,
	\end{align*}
	as desired.
	
	Since the tuple $(G'', L'', H'', p'')$ is nice, $G''$ admits a $p''$-moderate $(L'', H'')$-coloring. But $(G'', L'', H'', p'')$ is a reduct of $(G, H, L, p)$, so this coloring can be extended to a $p$-moderate $(L, H)$-coloring of $G$. Therefore, $G$ is $(L,H)$-colorable, as desired.
	
	\paragraph*{Acknowledgments}
	This work is supported by the Illinois Distinguished Fellowship.
	I am grateful to Alexandr Kostochka for drawing my attention to the notion of correspondence coloring and for helpful conversations. I am also grateful to the anonymous referees for their valuable comments.


\begin{thebibliography}{99}\setlength{\itemsep}{-0.001mm}
		\bibitem{Alon} N.~Alon. \emph{Degrees and choice numbers}. Random Structures \& Algorithms, Volume 16, 2000. Pages 364--368.
		
		\bibitem{AS} N.~Alon, J.H.~Spencer. \emph{The Probabilistic Method}. Wiley, New York, 1992.
		
		\bibitem{Borodin} O.~Borodin. \emph{Colorings of plane graphs: A survey}. Discrete Mathematics, Volume 313, 2013. Pages 517--539.
		
		\bibitem{Bernshteyn} A.~Bernshteyn. \emph{The Local Cut Lemma}. \href{http://arxiv.org/abs/1601.05481}{\texttt{arXiv:1601.05481}}, preprint, 2016.
		
		\bibitem{DP} Z.~Dvo\v{r}\'{a}k and L.~Postle. \emph{List-coloring embedded graphs without cycles of lengths $4$ to $8$}. \href{http://arxiv.org/abs/1508.03437}{\texttt{arXiv:1508.03437}}, preprint, 2015.
		
		\bibitem{ERT} P.~Erd\H{o}s, A.L.~Rubin and H.~Taylor. \emph{Choosability in graphs}. Proc. West Coast Conf. on Combinatorics, Graph Theory and Computing, Congressus Numerantium XXVI, 1979. Pages 125--157.
		
		\bibitem{Johansson} A.~Johansson. \emph{Asymptotic choice number for trangle free graphs}. Technical Report 91--95, DIMACS, 1996.
		
		\bibitem{MR}  M.~Molloy, B.~Reed. \emph{Graph colouring and the probabilistic method}. Algorithms and Combinatorics, Volume 23, Springer-Verlag, Berlin, 2002.
		
		\bibitem{Th1} C.~Thomassen. \emph{Every planar graph is $5$-choosable}. J. Combin. Theory, Ser. B, Volume 62, 1994. Pages 180--181.
		
		\bibitem{Th2} C.~Thomassen. \emph{$3$-list-coloring planar graphs of girth $5$}. J. Combin. Theory, Ser. B, Volume 64, 1995. Pages 101--107.
		
		\bibitem{Vizing} V.G.~Vizing. \emph{Vertex colorings with given colors} (In Russian). Diskret. Analiz., Volume 29, 1976. Pages 3--10.	
	\end{thebibliography}
\end{document}